\documentclass[11pt,a4paper,twoside,leqno]{amsart}
\usepackage[english]{babel}
\usepackage[foot]{amsaddr}
\usepackage[dvipsnames]{xcolor}
\definecolor{britishracinggreen}{rgb}{0.0, 0.26, 0.15}
\definecolor{cobalt}{rgb}{0.0, 0.28, 0.67}

\usepackage[utopia]{mathdesign}
    \DeclareSymbolFont{usualmathcal}{OMS}{cmsy}{m}{n}
    \DeclareSymbolFontAlphabet{\mathcal}{usualmathcal}

\usepackage[a4paper,top=3.5cm,bottom=3cm,left=3.5cm,
           right=3.5cm,bindingoffset=5mm]{geometry}
\usepackage{amsmath,amsthm}
\usepackage{indentfirst}
\usepackage[colorlinks,bookmarks]{hyperref} 
\usepackage{braket}
\usepackage{caption}
\usepackage{stmaryrd}
\usepackage{comment}
\usepackage{multirow}
\usepackage{booktabs}
\usepackage{ytableau}

\usepackage{microtype}

\numberwithin{equation}{section}
\renewenvironment{proof}{{\scshape Proof.}}{\qed}

\def\be{\begin{equation}}    
\def\ee{\end{equation}}
\def\bitem{\begin{itemize}}
\def\eitem{\end{itemize}}
\def\benum{\begin{enumerate}}
\def\eenum{\end{enumerate}}

\def\isom{\cong}  
\def\ra{\rightarrow}

\def\Var{\mathsf{Var}}

\def\St{\mathsf{St}}
\def\L{\mathbb L}
\def\A{\mathbb A}
\def\P{\mathbb P}
\def\Z{\mathbb Z}
\def\C{\mathbb C}

\def\G{\mathbb G}

\def\O{\mathscr O}

\DeclareMathOperator{\Rep}{Rep}

\DeclareMathOperator{\Spec}{Spec\,}

\DeclareMathOperator{\GL}{GL}

\DeclareMathOperator{\Hilb}{Hilb}

\DeclareMathOperator{\Aut}{Aut}

\DeclareMathOperator{\Hom}{Hom}
\DeclareMathOperator{\End}{End}

\DeclareMathOperator{\rk}{rk}



\newtheoremstyle{conv} 
        {4mm}
        {4mm}
        {\rmfamily}
        {4mm}
        {\itshape}
        {.}
        {1mm}
        {}
\theoremstyle{conv}

\newtheorem*{notation*}{Notation}

\newtheoremstyle{thm} 
        {4mm}
        {4mm}
        {\slshape}
        {4mm}
        {\scshape}
        {.}
        {1mm}
        {}
\theoremstyle{thm}
\newtheorem{teorr}{Theorem}
\newtheorem*{teo*}{Theorem}
\newtheorem{prop}{Proposition}[section]
\newtheorem{lemma}[prop]{Lemma}

\newtheorem{teo}[prop]{Theorem}

\theoremstyle{definition}
\newtheorem{example}[prop]{Example}
\newtheorem{definition}[prop]{Definition}
\newtheorem{remark}[prop]{Remark}

\usepackage{tikz}
\usepackage{tikz-cd}
\usepackage{rotating}

\usetikzlibrary{matrix,shapes,arrows,decorations.pathmorphing}
\tikzset{commutative diagrams/arrow style=math font}
\tikzset{commutative diagrams/.cd,
mysymbol/.style={start anchor=center,end anchor=center,draw=none}}

\title{On coherent sheaves of small length on the affine plane}
\author{Riccardo Moschetti}
\author{Andrea T. Ricolfi}
\email{riccardo.moschetti@uis.no,~andrea.ricolfi@uis.no}
\begin{document}

\begin{abstract}
We classify coherent modules on $k[x,y]$ of length at most $4$ and supported at the origin. We compare our calculation with the motivic class of the moduli stack parametrizing such modules, extracted from the Feit--Fine formula.
We observe that the natural torus action on this stack has finitely many fixed points, 
corresponding to connected skew Ferrers diagrams. 
\end{abstract}

\maketitle
    \begingroup
    \hypersetup{linkcolor=black}
    \tableofcontents
    \endgroup
    
\hypersetup{
            citecolor=cobalt,%
            linkcolor=britishracinggreen}

\section{Introduction}
Finitely generated modules over principal ideal domains satisfy a known structure theorem, and finitely generated modules over a Dedekind domain were characterized by Steinitz over a century ago~\cite{Steinitz1911}. In this paper we study modules of finite length over the polynomial ring
\[
A=k[x,y],
\]
where $k$ is an algebraically closed field of characteristic zero. These naturally correspond to coherent sheaves supported on finitely many points of the affine plane $\A^2$. We only consider modules entirely supported at the origin, as every other module is isomorphic to a direct sum of modules of this type. 

Among all such modules are artinian ring quotients of $A$, corresponding to points of the \emph{punctual Hilbert scheme}
\[
\Hilb^n(\A^2)_0\subset \Hilb^n(\A^2).
\]
It is known that there are infinitely many isomorphism types of artinian $k$-algebras of length $8$. Those of length at most $6$ have been completely described by Brian\c{c}on \cite{Bria1}.
Moreover, Poonen has classified isomorphism types of $k$-algebras of dimension up to $6$ and proves there are infinitely many types in dimension $n\geq 7$~\cite{Poo1}. 

The problem of classifying $A$-modules of finite length up to $A$-linear isomorphism, that we study here, is equivalent to the one of classifying pairs of commuting linear transformations on a finite dimensional vector space; the latter is known to contain the problem of classifying arbitrary tuples of commuting linear transformations, by work of Gelfand and Ponomarev~\cite{GelPon}.

We could not find a reference in the literature for the classification of isomorphism classes of $A$-modules of length $n>2$. We deal with length $n=3$ and $n=4$ in this work.

\subsection{Main result}
For simplicity, we state our classification in terms of (isomorphism classes of) indecomposable modules only, but see Tables \ref{Tbl:mod3} and \ref{Tbl:aut4} for complete lists including the decomposable ones. All modules are supported entirely on $\mathfrak m=(x,y)$, the ideal of the origin. Our main result is the following.

\begin{teorr}\label{thm:main}
The indecomposable modules of length $3$ are either structure sheaves or the distinguished module $\Hom_k(A/\mathfrak m^2,k)$.
In the length $4$ case, besides structure sheaves, there are two families $\mathcal F_1$ and $\mathcal F_2$ of indecomposable modules, both isomorphic to $\P^1$.
\end{teorr}

The stack of $A$-modules of length $n$ will be denoted $\mathcal C(n)$ throughout. We let
\[
\mathcal C(n)_0\subset \mathcal C(n)
\]
be the closed substack
parametrizing modules entirely supported at the origin.
Even though these stacks are not well understood in general, remarkably their motivic classes in the Grothendieck ring of stacks can be computed for arbitrary $n$, by means of the Feit--Fine formula \cite{FF1,BBS,BM15}. The knowledge of the motivic aspect of the theory is for us both a motivation to tackle the classification problem, and a way to check our results. More precisely, our strategy goes as follows. We stratify $\mathcal C(n)_0$ by locally closed substacks
\[
\mathcal X_r(n)\subset \mathcal C(n)_0, 
\]
each parametrizing modules $M$ such that $\dim_k M/\mathfrak m\cdot M=r$. In other words, we study modules $M$ by means of their discrete invariant
\[
r_M=\textrm{minimal number of generators of }M.
\]
We then analyze each stratum separately, and we compute its motivic class. Since we are inside a quotient stack, this requires us to compute all possible automorphism groups. To confirm our calculation we verify that the sum
\[
\sum_{r=1}^n\,\bigl[\mathcal X_r(n)\bigr]\in K_0(\St_k)
\]
reconstructs the class $[\mathcal C(n)_0]$ predicted by the Feit--Fine formula.

\medskip
In Section \ref{sec:puntifissi} we study the natural action of the torus $\mathbf T=\mathbb G_m^2$ on the moduli stack $\mathcal C(n)$ and we prove that it has finitely many fixed points, corresponding to certain types of skew Ferrers diagrams. We will also observe that the generating function for the numbers of indecomposable torus-fixed modules has a well-known combinatorial interpretation in terms of parallelogram polyominoes.

\section{Main strategy and conventions}\label{strategy}
Let $k$ be an algebraically closed field of characteristic zero. We let
\[
\mathfrak m\subset A=k[x,y]
\]
be the maximal ideal of the origin in $\A^2$.
All $A$-modules $M$ are assumed to be of finite length and entirely supported at $\mathfrak m$. Note that since the function
\[
\mathsf r:M\mapsto \dim_k M/\mathfrak m\cdot M
\]
is upper semi-continuous, there are well-defined locally closed substacks
\[
\mathcal X_r(n)\subset \mathcal C(n)_0
\]
parametrizing modules with $r$ as minimal number of generators.
The motivic class of $\mathcal X_r(n)$ makes sense, and we have a decomposition
\[
\bigl[\mathcal C(n)_0\bigr]=\sum_{r=1}^n\,\bigl[\mathcal X_r(n)\bigr]\in K_0(\St_k).
\]
See~\cite{EKStacks} for an introduction to the Grothendieck group of algebraic stacks.

If $M$ lies in $\mathcal X_r(n)$, we will find useful to fix a $k$-linear basis 
\[
\set{v_1,v_2,\dots,v_n}\subset M
\]
such that the first $r$ vectors generate the module, and $v_{r+1},\dots,v_n$ generate the submodule
\[
\mathfrak m\cdot M\subset M
\]
as a $k$-vector space.

\begin{lemma}\label{lemma:lemma1}
With the above choice of basis, $x\cdot v_i$ and $y\cdot v_i$ belong to $\mathfrak m\cdot M$.
\end{lemma}

\begin{proof}
Let $\pi:M\twoheadrightarrow  M/\mathfrak m\cdot M$ be the canonical projection. By our choice of basis, $v_{r+1},\dots,v_n$ form a $k$-basis of $\mathfrak m\cdot M$ and $M/\mathfrak m\cdot M$ is generated over $k$ by the images of $v_1,\dots,v_r$. Writing $x\cdot v_i=\sum_{j=1}^na_jv_j$ for some $a_j\in k$, from the relation
\[
0=x\cdot \pi(v_i)=\pi(x\cdot v_i)=\sum_{j=1}^na_j\pi(v_j)=\sum_{j=1}^ra_j\pi(v_j)
\]
one deduces that $a_j=0$ for $j=1,\dots,r$. Therefore $x\cdot v_i=\sum_{j=r+1}^na_jv_j$ belongs to $\mathfrak m\cdot M$, and similarly for $y\cdot v_i$.
\end{proof}

\medskip
The lemma says that after fixing a suitable basis for $M$ one can see multiplication by $x$ and $y$ as $k$-linear maps
\be\label{puttanazza}
\begin{tikzcd}
\langle v_1,\dots,v_r\rangle_k
\arrow[r, shift left]{}{x}
\arrow[r, shift right]{}[swap]{y}
& \mathfrak m\cdot M.
\end{tikzcd}
\ee
If the additional condition $x\cdot(y\cdot v_i)=y\cdot(x\cdot v_i)$ is fulfilled for $i=1,\dots,r$ then the two $k$-linear maps above characterize $M$. Such point of view will be essential when dealing with length $4$ modules satisfying $r_M=2$. Then our strategy will be to classify all pairs of linear maps as above, for each choice of length two submodule $\mathfrak m\cdot M\subset M$.

We anticipate here that we will sometimes describe modules by means of their ``multiplication table''. This is just a way to represent the action of $x$ and $y$ on a chosen basis $v_1,\dots,v_n$. The first $r$ entries of the table are to be filled in according to Lemma \ref{lemma:lemma1}, whereas the last $n-r$ are describing the submodule $\mathfrak m\cdot M\subset M$. Occasionally, we will encounter modules that can be visually represented as certain types of \emph{skew Ferrers diagrams}, see Example \ref{ex:distinguishedl3} below. In Section \ref{sec:puntifissi} we will see that these special modules are the finitely many fixed points of the natural torus action on the moduli stack $\mathcal C(n)$. For instance, as is well-known, classical Ferrers diagrams correspond to the fixed points of the torus action on $\Hilb^n(\A^2)$, studied in \cite{ESHilb}.
Recall that Ferrers diagrams (also called Young diagrams) correspond to ordinary partitions of integers, whereas a \emph{skew Ferrers diagram} is a difference of two Ferrers diagrams.

Our convention for (skew) Ferrers diagrams is to use the French notation; when a module can be represented by a skew Ferrers diagram, we understand multiplication by $x$ (resp.~$y$) in the module as shifting position to the right (resp.~to the top) in the diagram. The following example illustrates our conventions.

\begin{example} \label{ex:distinguishedl3}
Consider the module $M=A/\mathfrak m^2=k[x,y]/(x^2,xy,y^2)$. This is the unique non-curvilinear structure sheaf of length $3$, with natural $k$-basis $1,x,y$. The multiplication tables
\[
\begin{tabular}{cccc}
& $1$ & $x$ & $y$ \\
\toprule
$x\cdot$ & $x$ & $0$ & $0$ \\
$y\cdot$ & $y$ & $0$ & $0$ \\
\bottomrule
\end{tabular}
\qquad\qquad
\begin{tabular}{cccc}
& $1^\ast$ & $x^\ast$ & $y^\ast$ \\
\toprule
$x\cdot$ & $0$ & $1^\ast$ & $0$ \\
$y\cdot$ & $0$ & $0$ & $1^\ast$ \\
\bottomrule
\end{tabular}
\]
describe, respectively, the $A$-linear structure of $M$ and of its $k$-linear dual $M^\ast=\Hom_k(M,k)$. These tables can be represented as diagrams
\[
\ytableausetup{boxsize=normal}
  \begin{ytableau}
   y \\
   1 & x
  \end{ytableau}
\qquad\qquad\qquad
\ytableausetup{boxsize=normal}
  \begin{ytableau}
  x^\ast & 1^\ast \\
   \none & y^\ast
  \end{ytableau}
\]
where the first one is a classical Ferrers diagram and the second one is the \emph{skew} Ferrers diagram corresponding to $M^\ast$. We will see as a part of Theorem \ref{thm:main} that $M^\ast$ is the unique indecomposable module of length $3$ that is not a structure sheaf. In the following we will avoid writing the name of the generators inside the diagrams.
\end{example}

As a warm-up to illustrate our classification technique, we now describe all isomorphism types of $A$-modules belonging to the stratum $\mathcal X_{n-1}(n)\subset \mathcal C(n)_0$. Recall that by $r=r_M$ we mean the minimal number of generators of $M$.

\begin{prop}\label{prop:n,n-1}
Any module $M$ of length $n\geq 3$ with $r_M=n-1$ is isomorphic to $k^{n-3}\oplus N$, where $N$ has length $3$ and $r_N=2$.
\end{prop}

\begin{proof}
Any $A$-module $N$ of length $3$ with $r_N=2$ gives rise to
$M=k^{n-3}\oplus N$, with $r_M=n-1$. Conversely, if $M$ is generated over $A$ by $v_1,\dots,v_{n-1}$ and $v_n$ generates $\mathfrak m\cdot M=k$, then the multiplication table for $M$ looks like
\[
\begin{tabular}{cccccc} 
 & $v_1$ & $v_2$ & $\cdots$ & $v_{n-1}$ & $v_n$ \\
\toprule 
$x\cdot$ & $a_1v_n$ & $a_2v_n$ & $\cdots$ & $a_{n-1}v_n$ & $0$ \\
$y\cdot$ & $b_1v_n$ & $b_2v_n$ & $\cdots$ & $b_{n-1}v_n$ & $0$\\
\bottomrule
\end{tabular}
\]

\medskip
\noindent
where $a_i$ and $b_i$ are scalars in $k$. Up to relabeling the generators we can assume either $a_1$ or $b_1$ to be nonzero. We deal with the former case, since the latter is completely symmetric.

If $a_1\neq 0$ we can assume it is equal to $1$. Replacing $v_i$ by $v_i-a_iv_1$ for $i=2,\dots,n-1$ we get the multiplication table
\[
\begin{tabular}{cccccc} 
 & $v_1$ & $v_2$ & $\cdots$ &  $v_{n-1}$ & $v_n$ \\
\toprule 
$x\cdot$ & $v_n$ & $0$ & $\cdots$ & $0$  & $0$\\
$y\cdot$ & $b_1v_n$ & $b_2'v_n$ & $\cdots$  & $b_{n-1}'v_n$ & $0$\\
\bottomrule
\end{tabular}
\]

\medskip
\noindent
If $b_2'=0$ then $M=\langle v_2\rangle_k\oplus F$ for $F$ a module of length $n-1$, and the result follows by induction. If $b_2' \neq 0$, we can assume $b_2'=1$ and replace $v_i$ by $v_i-b_i'v_{2}$ for $i=3,\dots,n-1$. This yields 
\[
\begin{tabular}{ccccccc} 
 & $v_1$ & $v_2$ & $v_3$ & $\cdots$ & $v_{n-1}$ & $v_n$ \\
\toprule 
$x\cdot$ & $v_n$ & $0$ & $0$ & $\cdots$ &  $0$  &  $0$ \\
$y\cdot$ & $b_1v_n$ & $v_n$ & $0$ & $\cdots$ &  $0$  &  $0$ \\
\bottomrule
\end{tabular}
\]

\medskip
\noindent
so that $M=k^{n-3}\oplus F$ where $F$ is a length $3$ module generated by $v_1$ and $v_2$.
\medskip
\end{proof}

\section{Motivic interpretation}
The Grothendieck ring of algebraic stacks $K_0(\St_k)$ carries a \emph{power structure} naturally extending the one present on the classical Grothendieck ring of varieties, which is due to Gusein-Zade, Luengo
and Melle-Hern\'andez \cite{GLMps}. We refer to \cite{BM15} for more details. Define the generating functions
\[
\mathsf C(t)=\sum_{n=0}^\infty\, \bigl[\mathcal C(n)\bigr] t^n,
\qquad
\mathsf C_0(t)=\sum_{n=0}^\infty\, \bigl[\mathcal C(n)_0\bigr] t^n.
\]
As observed in \cite{BM15}, these power series are related, via the power structure on $K_0(\St_k)$, by $\mathsf C(t)=\mathsf C_0(t)^{\L^2}$. 
We now recall a formula for $\mathsf C(t)$ originally proved by Feit and Fine \cite{FF1} in the context of point counting over $\mathbb F_q$.

\begin{teo}\cite{BBS,BM15}
In $K_0(\St_\C)\llbracket t\rrbracket$, one has the formula
\[
\mathsf C(t)=\prod_{k=1}^\infty\prod_{m=1}^\infty (1-\L^{2-k}t^m)^{-1}.
\]
\end{teo}

\begin{remark}
The above relation, proved over $\C$ in the references given, holds in the Grothendieck ring of stacks over any algebraically closed field $k$. The main technical result needed for the proof in \cite{BM15}, besides the existence of the Jordan normal form, is that $K_0(\St_k)$ is isomorphic to the localization of $K_0(\Var_k)$ at the classes $\L$ and $\L^n-1$ for $n>0$. This is true over any field by \cite[Theorem $1.2$]{EKStacks}.
\end{remark}

The properties of the power structure \cite{BM15} allow one to deduce
\be\label{punctual0}
\mathsf C_0(t)=\mathsf C(t)^{\L^{-2}}=\prod_{k=1}^\infty\prod_{m=1}^\infty (1-\L^{-k}t^m)^{-1}.
\ee
Expanding the above series, one finds
\[
\mathsf C_0(t)=1+\frac{1}{\L-1}t+\left(\frac{1}{[\GL_2]}+\frac{\L+1}{\L(\L-1)}\right)t^2+\cdots
\]
The geometric interpretation of the first coefficients is clear: 
\bitem
\item [$(0)$] the motivic class of $\mathcal C(0)_0=\Spec k$ is just $1$;
\item [$(1)$] $1/(\L-1)$ is the motivic class of the stack $\mathcal C(1)_0=\textrm{B}\G_m$, which has only one point (corresponding to the module $k=A/\mathfrak m$), weighted by its automorphism group $\G_m$;
\item [$(2)$] the stack $\mathcal C(2)_0$ decomposes as $\mathcal X_2(2)=\textrm B\GL_2$ (corresponding to $k\oplus k$) union $\mathcal X_1(2)=\P^1/\G_a\rtimes\G_m$, where the projective line $\P^1=\P(\mathfrak m/\mathfrak m^2)$ represents the punctual Hilbert scheme $\Hilb^2(\A^2)_0$, which parametrizes structure sheaves of length $2$ supported at the origin (each having automorphism group $\G_a\rtimes \G_m$).
\eitem
 
\medskip
We aim at giving a similar interpretation of the next two coefficients of $\mathsf C_0(t)$, using our stratification. 

When we present a (substack of a) stratum $\mathcal X_r(n)$ as a quotient stack $Y/G$, we will say that any module belonging to this stratum has \emph{motivic contribution} the motive of the scheme $Y$.

\subsection{Automorphism groups}
The following general result on automorphism groups of quiver representations is going to help us compute all automorphism groups of $A$-modules of finite length.

\begin{teo}\cite[Prop.~$2.2.1$]{Brion1}\label{thm:autquiver}
Let $M$ be a finite dimensional representation of a quiver $Q$. Then $\Aut(M)$ is a connected linear algebraic group, with a decomposition 
\[
\Aut(M)=U\rtimes \prod_{i=1}^s\GL_{m_i},
\]
where $U$ is a closed normal unipotent subgroup and $m_1,\dots,m_s$ are the multiplicities of the indecomposable summands of $M$.
\end{teo}

We apply this result to the quiver $Q$ consisting of one node and two loops. The category $\Rep(Q)$ of representations of $Q$ is equivalent to the category of left modules over the path algebra of $Q$, which is the non-commutative algebra $kQ=k\braket{x,y}$. We need to consider the quiver with relations $(Q,I)$, where $I\subset kQ$ is the two-sided ideal generated by the single commutator $xy-yx$. Then the category of representations of $(Q,I)$ is a full subcategory of $\Rep(Q)$, naturally equivalent to the category of modules over $kQ/I=k[x,y]=A$.

Theorem \ref{thm:autquiver} implies in particular that, if $M$ is an $A$-module of finite length, its automorphism group is a \emph{special} algebraic group. Indeed, the $\GL$ factors are themselves special, every unipotent group in characteristic zero is an iterated extensions of copies of $\G_a$ (which is special), and any (semidirect) product of special groups is special. This fact is crucial for us: the Grothendieck ring $K_0(\St_k)$ can be characterized as the localization of $K_0(\Var_k)$ at the classes of special algebraic groups, and the upshot is that when a variety $Y$ is acted on by a special group $G$, the motivic class of the quotient stack $\mathscr Y=Y/G$ can be computed as
\[
\bigl[\mathscr Y\bigr]=\bigl[Y\bigr]\,\big/\,\bigl[G\bigr]\in K_0(\St_k).
\]

In Tables \ref{Tbl:mod3} and \ref{Tbl:aut4}, the column indicating the motive of the automorphism groups is obtained directly from Theorem~\ref{thm:autquiver}. Note that in order to compute this class, one only needs the indecomposable factors of the module, and the dimension of the automorphism group $\Aut_A(M)$. The latter is an elementary calculation and easily gives the number of copies of $\G_a$ appearing in the unipotent factor $U$.

As an example, consider the stack of coherent sheaves of length $n$ on affine space $\A^d$, supported at the origin. By Theorem~\ref{thm:autquiver}, its locally closed substack parametrizing structure sheaves has motivic class
\[
\frac{\bigl[\Hilb^n(\A^d)_0\bigr]}{\L^{n-1}(\L-1)}.
\]
Indeed, $\O_Z$ is indecomposable and $\Aut(\O_Z)$ has dimension $n$ for all fat points $Z\subset \A^d$, since an automorphism is determined by the image of $1\in \O_Z$. If $d=2$, this computes the class
\be\label{eqn:strshcontr}
\bigl[\mathcal X_1(n)\bigr]=\frac{\bigl[\Hilb^n(\A^2)_0\bigr]}{\L^{n-1}(\L-1)}.
\ee
To obtain the motive of the punctual Hilbert scheme, one can use the expansion
\be\label{eqn:punctstruct}
\sum_{n=0}^\infty\, \bigl[\Hilb^n(\A^2)_0\bigr]t^n=\prod_{m\geq 1}(1-\L^{m-1}t^m)^{-1}.
\ee
This follows directly from the motivic version of G\"{o}ttsche's formula \cite{Gott1}, as it is also explained in \cite{GLMps}.

\begin{example}
The curvilinear locus inside the punctual Hilbert scheme is known to be a dense open subset, fibred over $\P^1$ with fibre $\A^{n-2}$~\cite{Bria1}. Hence the motivic class $[\Hilb^n(\A^2)_0]$ always decomposes as $\L^{n-2}(\L+1)$ plus the class of the non-curvilinear locus. If $n=3$, the latter is just a single point corresponding to $A/\mathfrak m^2$, hence we get
\[
\bigl[\Hilb^3(\A^2)_0\bigr]=\L(\L+1)+1.
\]
For higher $n$, one has to expand~\eqref{eqn:punctstruct} in order to extract the motive of the punctual Hilbert scheme.
\end{example}

\section{Modules of length three}
We have three strata $\mathcal X_r(3)\subset \mathcal C(3)_0$.
Structure sheaves correspond to $r=1$ and contribute
\[
\bigl[\mathcal X_1(3)\bigr]=\frac{\L(\L+1)+1}{\L^2(\L-1)}
\]
by formula \eqref{eqn:strshcontr}. We have $\mathcal X_3(3)=\textrm{B}\GL_3$, and we compute $\mathcal X_2(3)$ in the next proposition.

\begin{prop}\label{prop:l3-2}
The modules of length $3$ having $r=2$ are $(A/\mathfrak m^2)^\ast$ and those of the form $k\oplus \O_Z$, where $Z\subset \A^2$ is a subscheme of length $2$. The latter form a family isomorphic to $\P^1$.
\end{prop}

\begin{proof}
This can be extracted from the proof of Proposition \ref{prop:n,n-1}. Indeed, the decomposable modules (necessarily of the form $k\oplus \O_Z$) occur when $b_2'=0$, and the only new module corresponds to the final table, which represents (independently upon the choice of $b_1$) the indecomposable module $(A/\mathfrak m^2)^\ast$.
\end{proof}

\medskip
This proves the part of Theorem \ref{thm:main} concerning $n=3$.
The skew Ferrers diagrams representing a length $3$ module are precisely

\medskip
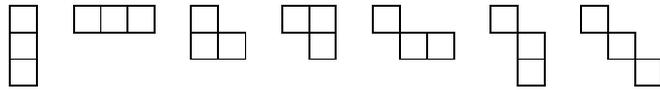
\[
\tiny{\ydiagram{1,1,1}}\qquad\tiny{\ydiagram{3}}\qquad \tiny{\ydiagram{1,2}}
\qquad \tiny{\ydiagram{2,1+1}} \qquad \tiny{\ydiagram{1,1+2}}\qquad \tiny{\ydiagram{1,1+1,1+1}}\qquad \tiny{\ydiagram{1,1+1,2+1}}
\]

\medskip
\noindent
where the first $4$ are indecomposable.
The automorphism groups of all modules are computed through Theorem \ref{thm:autquiver}. See Table \ref{Tbl:mod3} for the complete list.

\begin{table}[ht]
\centering
\begin{tabular}{cccc}
 $r$ & $\mathcal C(3)_0$ &  $[\Aut_A(M)]$ & Motivic contribution \\
\toprule
$1$ & $\O_Z$ & $\L^2(\L-1)$ & $\L(\L+1)+1$ \\ \\
$2$ & $(A/\mathfrak m^2)^\ast$ & $\L^2(\L-1)$ & $1$ \\ \\
$2$ & $k\oplus \O_Z$ & $\L^3(\L-1)^2$ & $\L+1$ \\ \\ 
$3$ & $k^{\oplus 3}$ & $[\GL_3]$ & $1$ \\ 
\bottomrule
\end{tabular} 

\medskip
\caption{All $k[x,y]$-modules of length $3$ supported at the origin, along with the class of their automorphism group and their motivic contribution. The first two rows describe indecomposable modules.}\label{Tbl:mod3}
\end{table}

\begin{remark}
The coefficient of $t^3$ in \eqref{punctual0} can be computed to be
\[
\bigl[\mathcal C(3)_0\bigr]=\frac{1}{[\GL_3]}\left(\L^8+\L^7+\L^6-\L^5-\L^4\right).
\]
The sum of the motives in the rightmost column of Table \ref{Tbl:mod3}, each divided by the motive of the corresponding automorphism group, recovers precisely this class, confirming our calculation.
\end{remark}

\section{Modules of length four}
We need to analyze the strata $\mathcal X_r(4)\subset \mathcal C(4)_0$ for $r=1$, $2$, $3$, $4$.
Expanding \eqref{eqn:punctstruct} we find
\[
\bigl[\Hilb^4(\A^2)_0\bigr]=\L^3+2\L^2+\L+1,
\]
and this determines the class of $\mathcal X_1(4)$ through \eqref{eqn:strshcontr}.
On the other hand, $\mathcal X_4(4)=\textrm{B}\GL_4$, and combining Propositions~\ref{prop:n,n-1} and \ref{prop:l3-2} with one another completely describes the stratum $\mathcal X_3(4)$ as follows.

\begin{prop}\label{cor:dec7183}
Let $M$ be a module of length $4$ with $r_M=3$. Then either $M\isom k\oplus (A/\mathfrak m^2)^\ast$, or $M\isom k^2\oplus \O_Z$, where $\O_Z$ is a structure sheaf of length $2$.
\end{prop}

The three skew Ferrers diagrams determined by Proposition~\ref{cor:dec7183} are depicted in Figure~\ref{fig:dec871} below.

\begin{figure}[!ht]
\begin{tikzpicture}[scale=0.75]
\draw (2,3) to (6,3);
\node[anchor=west, right] at (1.7,3)
   {\tiny{$\bullet$}};
\node[anchor=west, right] at (5.7,3)
   {\tiny{$\bullet$}};
\node[anchor=west, right] at (3.7,3.4)
   {$\P^1$};   
\node[anchor=west, right] at (1,2)
{\tiny{\ydiagram{1,1+1,2+2}}};
\node[anchor=west, right] at (5.3,1.8)
{\tiny{\ydiagram{1,1+1,2+1,2+1}}};
\node[anchor=west, right] at (9.3,2.2)
{\tiny{\ydiagram{1,1+2,2+1}}};
\end{tikzpicture}\caption{The decomposable modules of Proposition~\ref{cor:dec7183}. The $\P^1$ represents the family $k^2\oplus \O_Z$, whereas the isolated diagram represents $k\oplus (A/\mathfrak m^2)^\ast$.}\label{fig:dec871}
\end{figure}
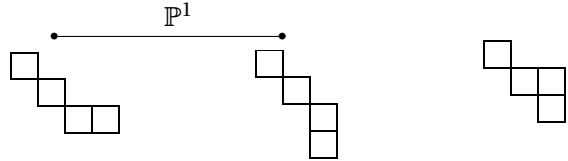

It remains to identify the stratum 
\[
\mathcal X_2(4)\subset \mathcal C(4)_0
\]
corresponding to modules with $\dim_k M/\mathfrak m\cdot M=2$.
These come in two types: either
\bitem
\item [(a)] $\mathfrak m\cdot M=\O_{Z}$, a structure sheaf of length $2$, or
\item [(b)] $\mathfrak m\cdot M=k\oplus k$.
\eitem
The families $\mathcal F_1$ and $\mathcal F_2$ of indecomposable modules mentioned in Theorem~\ref{thm:main} will arise from case (a) and (b), treated in Proposition \ref{prop:indec1} and Proposition \ref{prop:indec2} respectively.

According to \eqref{puttanazza}, we need to understand the space of pairs of $k$-linear maps
\begin{equation*}
\begin{tikzcd}
\langle v_1,v_2\rangle_k
\arrow[r, shift left]{}{A_x}
\arrow[r, shift right]{}[swap]{A_y}
& \mathfrak m\cdot M
\end{tikzcd}
\textrm{ satisfying }x\cdot(A_y v_i)=y\cdot (A_x v_i)  \,\textrm{ for } i=1,2.
\end{equation*}
Here $A_x$ and $A_y$ are two by two matrices corresponding to multiplication by $x$ and $y$ restricted to the $A$-linear generators $v_1$ and $v_2$. We need to consider the above data up to the equivalence relation that identifies pairs of matrices that give rise to isomorphic modules. This equivalence relation is determined in Lemma \ref{equivrelation} below in the case where $\mathfrak m\cdot M=k\oplus k$.

\medskip
Case (a) above is completely solved by the following result.

\begin{prop}\label{prop:indec1}
The decomposable modules with $r=2$ and such that $\mathfrak m\cdot M$ is a structure sheaf of length $2$ form an $\A^1$-fibration over $\P^1$, hence with motivic class $\L(\L+1)$. The indecomposable ones form a family $\mathcal F_1$ isomorphic to $\P^1$.
\end{prop}

\begin{proof}
If $M$ is generated as a $k$-vector space by $\{ v_1,v_2,v_3,v_4 \}$, we can assume the $k$-linear generators of $\mathfrak m\cdot M=A/(x^2,y-tx)$ to be $v_3=1$ and $v_4=x$. Here we have fixed $t\in\A^1=\P^1\setminus \{\infty\}$; then, after imposing the relations $x\cdot(y \cdot v_i) = y\cdot(x \cdot v_i)$ for $i=1,2$, the multiplication table for $M$ is
\[
\begin{tabular}{ccccc}
 & $v_1$  & $v_2$  & $v_3$ & $v_4$\\ 
\toprule
 $x\cdot$  & $a_1v_3+b_1v_4$ & $a_2v_3+b_2v_4$ & $v_4$  & $0$ \\
 $y \cdot$ & $a_1tv_3+c_1v_4$ & $a_2tv_3+c_2v_4$ & $tv_4$ & $0$ \\
\bottomrule
\end{tabular}
\]

\medskip
\noindent
and since $M$ is generated as an $A$-module by $v_1$ and $v_2$, we can assume $a_1=1$. Then, we may assume $a_2=0$ by replacing $v_2$ with $=v_2-a_2v_1$. Since $x \cdot v_3=v_4$, we can assume $b_1=b_2=0$ by replacing $v_1$ and $v_2$ with $v_1-b_1v_3$ and $v_2-b_2v_3$ respectively. This yields the multiplication table
\[
\begin{tabular}{ccccc}
 & $v_1$  & $v_2$  & $v_3$ & $v_4$\\ 
\toprule
 $x\cdot$  & $v_3$ & $0$ & $v_4$  & $0$ \\
 $y \cdot$ & $tv_3+zv_4$ & $c_2'v_4$ & $tv_4$  & $0$ \\
\bottomrule
\end{tabular}
\]

\medskip
\noindent
where $z$ and $c_2'$ arise from the above changes of basis. Here we distinguish between two cases: either $c_2'=0$ or $c_2'\neq 0$. 
In the former case we obtain, for each $t \in \A^1$, a family of decomposable modules parametrized by $z \in \A^1$. 
This family extends to the whole $\P^1$ of double points, so that the full family is an $\A^1$-fibration over $\P^1$.

On the other hand, if $c_2'\neq 0$, we may assume $c_2'=1$. Replacing $v_1$ by $v_1-zv_2$, we may also assume $z=0$, so for each $t \in \A^1$ we have exactly one indecomposable module: its multiplication table is
\begin{equation} \label{Eqn:4.2.Struct.Indec}
\begin{tabular}{ccccc}
 & $v_1$  & $v_2$  & $v_3$ & $v_4$\\ 
\toprule
 $x\cdot$  & $v_3$ & $0$ & $v_4$  & $0$ \\
 $y \cdot$ & $tv_3$ & $v_4$ & $tv_4$ & $0$ \\
\bottomrule
\end{tabular}
\end{equation}
and this family again extends to $t=\infty$ giving a family $\mathcal F_1\isom\P^1$.
\end{proof}

\medskip

\begin{figure}[!ht]
\begin{tikzpicture}[scale=0.75]
\node[anchor=west] at (1.7,3.4)
   {$0$};
\node[anchor=west] at (5.6,3.3)
   {$\infty$};
\draw (2,3) to (6,3);
\node[anchor=west, right] at (1.7,3)
   {\tiny{$\bullet$}};
\node[anchor=west, right] at (5.7,3)
   {\tiny{$\bullet$}};
\node[anchor=west, right] at (3.7,3.4)
   {$\P^1$};   
\node[anchor=west, right] at (1,2)
{\tiny{\ydiagram{3,3+1}}};
\node[anchor=west, right] at (5.3,1.8)
{\tiny{\ydiagram{1,1+1,1+1,1+1}}};
\node[anchor=west] at (8.7,3.4)
   {$0$};
\node[anchor=west] at (12.6,3.3)
   {$\infty$};
\draw (9,3) to (13,3);
\node[anchor=west, right] at (8.7,3)
   {\tiny{$\bullet$}};
\node[anchor=west, right] at (12.7,3)
   {\tiny{$\bullet$}};
\node[anchor=west, right] at (10.7,3.4)
   {$\mathcal F_1$};   
\node[anchor=west, right] at (8,2)
{\tiny{\ydiagram{3,2+1}}};
\node[anchor=west, right] at (12.3,1.8)
{\tiny{\ydiagram{2,1+1,1+1}}};
\end{tikzpicture} \caption{The left picture represents the $z=0$ slice of the family of decomposable modules of Proposition \ref{prop:indec1}. The right picture is the family of indecomposable modules given by \eqref{Eqn:4.2.Struct.Indec}.}\label{fig:jhgwjhgwf}
\end{figure}
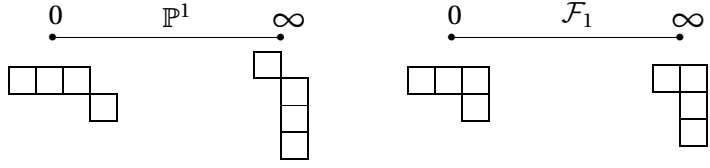

\begin{remark}
The decomposable modules of Proposition~\ref{prop:indec1} are those of the form $k\oplus \O_Z$, where $Z\subset \A^2$ is a \emph{curvilinear} subscheme of length $3$. The motive $\L(\L+1)$ is indeed the class of the curvilinear locus inside $\Hilb^3(\A^2)_0$. The family $\mathcal F_1$ parametrizes the $k$-linear duals $(A/I)^\ast$, where $I\subset k[x,y]$ is a non-complete-intersection ideal. This family of subschemes of the plane, parametrized by the $\P^1$ of linear forms on $k^2$, was studied by Brian\c{c}on \cite{Bria1}.
\end{remark}

It remains to treat case (b), so from now on we assume our modules $M$ satisfy
\[
\mathfrak m\cdot M = k\oplus k.
\]
This assumption makes the relations $x\cdot (y\cdot v_i)=y\cdot (x\cdot v_i)$ vacuous. We thus want to describe the quasi-affine variety
\be\label{kmaps}
U=\Set{k\textrm{-linear maps }
A_x,A_y:\langle v_1, v_2 \rangle_k\rightrightarrows k\oplus k|\rk(A_x\,A_y)=2}
\ee
up to a suitable group action. Here $(A_x A_y)$ is the $2\times 4$ matrix obtained by juxtaposing the two given square matrices. The rank condition comes from the requirement $r=2$. The next result characterizes pairs of matrices producing isomorphic modules.

\begin{lemma}\label{equivrelation}
Two pairs of matrices $(A_x, A_y)$ and $(B_x, B_y)$ as above give rise to isomorphic $A$-modules if and only if 
$$H A_x K = B_x, \qquad H A_y K = B_y$$
for some $H$, $K\in\GL_2$.
\end{lemma}

\begin{proof}
An isomorphism of modules can be identified with a matrix in $\GL_4$, which we write in block form as 
$$\begin{pmatrix}
W & X \\
Y & Z
\end{pmatrix}.$$
The matrix $X$ describes a mapping $\langle v_3, v_4 \rangle_k \to \langle v_1, v_2 \rangle_k$. Since the submodule generated by $v_3$ and $v_4$ is $k^2$, then $x \cdot v_3=y \cdot v_3 = x \cdot v_4 = y \cdot v_4=0$. Then, the matrix $X$ contributes trivially to the isomorphism, so we can assume $X$ is the zero matrix. On the other hand, the matrix $Y$ describes a component mapping $\langle v_1, v_2 \rangle_k \to \langle v_3, v_4 \rangle_k$. Since we are assuming $v_1$ and $v_2$ to be generators, $Y$ must be the zero matrix. It follows than $W$ and $Z$ belong to $\GL_2$. A direct computation shows that $K=W$ and $H=Z^{-1}$.
\end{proof}

\medskip
In what follows, we study the quotient stack
\[
U/\GL_2\times \GL_2
\]
which gives a presentation of the substack of $\mathcal X_2(4)$ representing modules of type (b).

\begin{prop}\label{prop:indec2}
The indecomposable modules of length $4$ with $\mathfrak m\cdot M=k\oplus k$ form a family $\mathcal F_2$ isomorphic to $\P^1$.
\end{prop}
\begin{proof}
The condition $\rk(A_x\,A_y)=2$ from \eqref{kmaps} says that either
\bitem
\item [$(1)$] $\rk A_x=\rk A_y=1$, or
\item [$(2)$] either $A_x$ or $A_y$ is invertible.
\eitem
In the first case, up to the action of $\GL_2\times \GL_2$ we can assume
$$A_x=\begin{pmatrix} 1 & 0 \\ 0 & 0 \end{pmatrix}.$$
At this point the  multiplication table of a module with $\mathfrak m\cdot M=k\oplus k$ looks like
\[
\begin{tabular}{ccccc}
 & $v_1$  & $v_2$  & $v_3$ & $v_4$\\ 
\toprule
 $x \cdot$ & $v_3$ & $0$ & $0$ & $0$ \\
 $y\cdot$ & $\alpha v_3+\beta v_4$ & $\gamma v_3 + \delta v_4$ & $0$ & $0$ \\
\bottomrule
\end{tabular}
\]

\medskip
\noindent
and the conditions $\rk A_x=\rk A_y=1$ and $\rk(A_x\,A_y)=2$ ensure that $\alpha \delta = \beta \gamma$ and $(\beta, \delta) \neq (0,0)$, respectively. Assume $\delta \neq 0$. Then by replacing $v_4$ with $ \gamma v_3 + \delta v_4$ we can assume $\gamma=0$, $\delta=1$ and therefore $\alpha=0$. Further replacing $v_1$ by $v_1-\beta b_2$ yields the table
\begin{equation} \label{Eqn:4.2.k2.X1Y1a}
\begin{tabular}{ccccc}
 & $v_1$  & $v_2$  & $v_3$ & $v_4$\\ 
\toprule
 $x \cdot$ & $v_3$ & $0$ & $0$  & $0$ \\
 $y\cdot$  & $0$ & $v_4$ & $0$  & $0$ \\
\bottomrule
 \end{tabular}
\end{equation}

\medskip
\noindent
which is a single module.
Similarly, if $\beta \neq 0$, replacing $v_4$ by $\alpha v_3 + \beta v_4$ we can assume $\alpha=0$, $\beta=1$ and therefore $\gamma=0$. Further replacing $v_2$ by $v_2-\delta b_1$ we get
\begin{equation} \label{Eqn:4.2.k2.X1Y1b}
\begin{tabular}{ccccc}
 & $v_1$  & $v_2$  & $v_3$ & $v_4$\\ 
\toprule
 $x \cdot$ & $v_3$ & $0$ & $0$  & $0$ \\
 $y\cdot$  & $v_4$ & $0$ & $0$  & $0$ \\
 \bottomrule
 \end{tabular}
\end{equation}

\medskip
\noindent
showing that case $(1)$ contributes only two isomorphism classes of modules, both decomposable and representable by skew Ferrers diagrams (see Figure~\ref{fig:jhgwjhgwf2}). 

We are left to deal with the loci in \eqref{kmaps} where $A_x$ is invertible and the one where $A_y$ is invertible. These are isomorphic along their common intersection. 
If, say, $A_x$ is invertible, by the action of $\GL_2\times \GL_2$ described in Lemma \ref{equivrelation}, we may assume $A_x$ is the identity matrix and $A_y$ is in Jordan form. If $A_y$ is not diagonalizable and has eigenvalue $\eta\in\A^1$, we get
\[
A_x=
\begin{pmatrix}
1 & 0\\
0 & 1
\end{pmatrix},\qquad A_y=
\begin{pmatrix}
\eta & 1\\
0 & \eta
\end{pmatrix}.
\]
Joining this family with the module represented by the pair 
\[
A_x=
\begin{pmatrix}
0 & 1\\
0 & 0
\end{pmatrix},\qquad 
A_y=
\begin{pmatrix}
1 & 0\\
0 & 1
\end{pmatrix}
\]
gives a family of indecomposable modules $\mathcal F_2$ parametrized by $\P^1$. This exhausts the non diagonalizable case, and all other modules are decomposable: for instance, if $A_x$ is invertible and $A_y$ is diagonalizable with eigenvalues $(\lambda,\mu)$, we obtain the module
\begin{equation*} 
\begin{array}{ccccc}
        &v_1  &v_2  &v_3   & v_4\\ \toprule
 x \cdot & v_3 & v_4 & 0  & 0 \\
 y\cdot  & \lambda v_3 & \mu v_4 & 0  & 0 \\
 \bottomrule
 \end{array}
\end{equation*}

\medskip
\noindent
which is the direct sum of two structure sheaves of length $2$. 
\end{proof}

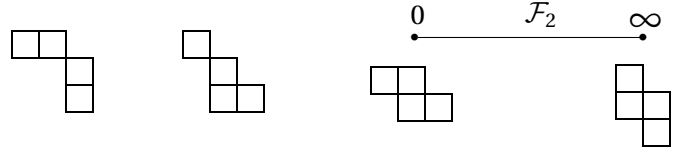
\begin{figure}[!ht]
\begin{tikzpicture}[scale=0.75]
\node[anchor=west] at (1.7,2.4)
   {\tiny{\ydiagram{2,2+1,2+1}}};
\node[anchor=west] at (4.7,2.4)   
   {\tiny{\ydiagram{1,1+1,1+2}}};
\node[anchor=west] at (8.7,3.4)
   {$0$};
\node[anchor=west] at (12.5,3.3)
   {$\infty$};
\draw (9,3) to (13,3);
\node[anchor=west, right] at (8.7,3)
   {\tiny{$\bullet$}};
\node[anchor=west, right] at (12.7,3)
   {\tiny{$\bullet$}};
\node[anchor=west, right] at (10.7,3.4)
   {$\mathcal F_2$};   
\node[anchor=west, right] at (8,2)
{\tiny{\ydiagram{2,1+2}}};
\node[anchor=west, right] at (12.3,1.8)
{\tiny{\ydiagram{1,2,1+1}}};
\end{tikzpicture} \caption{The left picture represents the two isolated modules ~\eqref{Eqn:4.2.k2.X1Y1a} and~\eqref{Eqn:4.2.k2.X1Y1b}. The right picture describes the family $\mathcal F_2=\P^1$ of indecomposable modules found in Proposition \ref{prop:indec2}.}\label{fig:jhgwjhgwf2}
\end{figure}

Theorem \ref{thm:main} follows combining Proposition \ref{prop:indec1} and Proposition \ref{prop:indec2} with one another.

\medskip
For completeness, let us finish the classification of decomposable modules. By the proof of the previous proposition, we are left to consider the locus, in $U$, where $A_x$ is invertible and $A_y$ is diagonalizable; this glues to the locus where $A_y$ is invertible and $A_x$ is diagonalizable. Up to the action of $\GL_2\times \GL_2$, when one between $A_x$ and $A_y$ is invertible, we can always assume the other to be in Jordan normal form.

Joining the locus where $A_x$ is invertible and $A_y$ has two equal eigenvalues with the module represented by the pair
\[
A_x=
\begin{pmatrix}
0 & 0\\
0 & 0
\end{pmatrix},\qquad 
A_y=
\begin{pmatrix}
1 & 0\\
0 & 1
\end{pmatrix}
\]
gives a family of decomposable modules
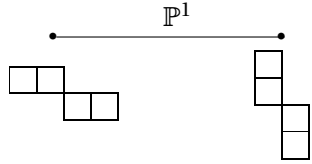
\[
\begin{tikzpicture}[scale=0.75]
\draw (2,3) to (6,3);
\node[anchor=west, right] at (1.7,3)
   {\tiny{$\bullet$}};
\node[anchor=west, right] at (5.7,3)
   {\tiny{$\bullet$}};
\node[anchor=west, right] at (3.7,3.4)
   {$\P^1$};   
\node[anchor=west, right] at (1,2)
	{\tiny{\ydiagram{2,2+2}}};
\node[anchor=west, right] at (5.3,1.8)
	{\tiny{\ydiagram{1,1,1+1,1+1}}};
\end{tikzpicture} 
\]
representing all length $4$ modules of the form $\O_Z\oplus \O_Z$, where $Z\subset \A^2$ is a subscheme of length $2$. These have automorphism group
\[
\G_a^4\rtimes \GL_2,
\]
which distinguishes them from the decomposables of the form $\O_Z\oplus\O_{Z'}$ with $Z\neq Z'$ two subschemes of length $2$. The latter indeed have automorphism group
\[
\G_a^4\rtimes \G_m^2.
\]
We have already encountered a module of this type, namely 
\be\label{prendilotutto}
A/(y,x^2)\oplus A/(x,y^2),
\ee
in the leftmost diagram of Figure \ref{fig:jhgwjhgwf2}. The other sums of (different) structure sheaves of length $2$ arise by considering the remaining types of pairs $(A_x,A_y)$ up to $\GL_2\times \GL_2$. More precisely, we have the locus where $A_x$ is invertible and $A_y$ has distinct eigenvalues $\lambda\neq \mu$, and finally the $\G_m$ of modules represented by matrices
\be\label{mangiacazzi}
A_x=
\begin{pmatrix}
0 & 0\\
0 & \nu
\end{pmatrix},
\qquad 
A_y=
\begin{pmatrix}
1 & 0\\
0 & 1
\end{pmatrix},\qquad \nu\neq 0.
\ee
We now need to compute the motivic contribution of this family.

\begin{lemma} \label{lemma:fraction}
The motivic contribution of the modules $\O_{Z} \oplus \O_{Z'}$, composed by the direct sum of two distinct structure sheaves of length $2$ is
\[
\frac{\L(\L^2+1)}{\L+1}.
\]
\end{lemma}
\begin{proof}
It is clear that \eqref{prendilotutto} and \eqref{mangiacazzi} together contribute $\L$. The remaining locus parametrizes Jordan forms of matrices with two distinct eigenvalues, namely
\[
\left\{(\lambda,\mu)\in \A^1 \times \A^1 \,|\, \lambda \neq \mu\right\} \big/\, \Z_2.
\]
We compute the motivic class $\xi$ of this locus formally, decomposing the motive of $\End(k^2)$ according to the Jordan type. We obtain the identity
\[
\L^4=\xi\cdot \frac{[\GL_2]}{(\L-1)^2}+\L\cdot \frac{[\GL_2]}{\GL_2} + \L\frac{[\GL_2]}{\L(\L-1)}
\]
where each ``fraction'' describes the orbit of a given Jordan form.
The middle $\L$ parametrizes matrices of the form $\lambda\cdot \textrm{Id}$, and similarly the last $\L$ corresponds to non-diagonalizable Jordan forms. After solving for $\xi$, the total contribution is 
\begin{equation*}
\xi+\L=\frac{\L^3-\L^2}{\L+1}+\L=\frac{\L(\L^2+1)}{\L+1},
\end{equation*}
proving the lemma.
\end{proof}

\medskip
We summarize in Table \ref{Tbl:aut4} all families of modules of length $4$. The automorphism groups of all modules are computed through Theorem~\ref{thm:autquiver}.

\begin{table}[!ht]
\centering
\begin{tabular}{cccc}
 $r$ & $\mathcal C(4)_0$ &  $[\Aut_A(M)]$ & Motivic contribution \\
\toprule
 $1$ & $\O_Z$ &  $\L^3(\L-1)$ & $\L^3+2\L^2+\L+1$    \\ \\
 $2$ & $\mathcal F_1$ & $\L^3(\L-1)$ & $\L+1$  \\ \\
 $2$ & $\mathcal F_2$ & $\L^5(\L-1)$ & $\L+1$  \\ \\
 $2$ & $k\oplus A/\mathfrak m^2$ & $\L^5(\L-1)^2$ & $1$  \\ \\
 $2$ & $k \oplus \O_Z$, $Z$ curvilinear & $\L^4(\L-1)^2$ & $\L(\L+1)$  \\ \\
 $2$ & $\O_Z \oplus \O_Z$ & $\L^4[\GL_2]$ & $\L+1$  \\ \\
 $2$ & $\O_Z \oplus \O_{Z'}$, $Z \neq Z'$ & $\L^4(\L-1)^2$ & $\displaystyle{\frac{\L(\L^2+1)}{\L+1}}$  \\ \\
 $3$ & $k^2\oplus \O_{Z}$ & $\L^5(\L-1)[\GL_2]$ & $\L+1$  \\ \\
 $3$ & $k\oplus (A/\mathfrak m^2)^\ast$ & $\L^5(\L-1)^2$ & $1$ \\ \\
 $4$ & $k^4$ & $[\GL_4]$ & $1$\\ 
\bottomrule
\end{tabular}

\medskip
\caption{All length $4$ modules supported at the origin, along with the class of their automorphism group and the corresponding motivic contribution. The first three rows describe indecomposable modules.}
\label{Tbl:aut4}
\end{table}

\begin{remark}
The coefficient of $t^4$ in \eqref{punctual0} can be computed to be
\[
\bigl[\mathcal C(4)_0\bigr]=\frac{1}{[\GL_4]}\left(\L^{15}+2\L^{14}+\L^{13}+\L^{12}-2\L^{11}-2\L^{10}-\L^9+\L^7\right).
\]
The sum of the motives in the rightmost column of Table \ref{Tbl:aut4}, each divided by the motive of the corresponding automorphism group, recovers precisely this class, confirming our calculation.
\end{remark}

\section{Torus action}\label{sec:puntifissi}
Let $V$ be an $n$-dimensional vector space, and consider the \emph{commuting variety}
\[
C_n=\Set{(X,Y)\in \End(V)^2|[X,Y]=0}\subset \End(V)^2.
\]
It contains the closed subscheme $N_n$ of pairs of commuting \emph{nilpotent} matrices. The group $\GL_n$ acts on these spaces by simultaneous conjugation, and the 
closed immersion of quotient stacks
\[
N_n/\GL_n\subset C_n/\GL_n
\]
is equivalent to the closed immersion
\[
\mathcal C(n)_0\subset \mathcal C(n)
\]
of the stack of coherent sheaves supported at the origin inside the full stack of coherent sheaves of length $n$.
The natural action of the torus $\mathbf T=\mathbb G_m^2$ on $\A^2$, given by rescaling coordinates,
\[
(t_1,t_2)\cdot (x,y)=(t_1x,t_2y),
\]
can be lifted to $\GL_n$-equivariant actions on $C_n$ and $N_n$. This gives an induced $\mathbf T$-action on $\mathcal C(n)$, leaving $\mathcal C(n)_0$ invariant. We show in Proposition \ref{prop:torusfix} below that this torus action has finitely many fixed points. This will finally make precise the connection with skew Ferrers diagrams, which we used as a mere graphical representation so far.

Recall that a skew Ferrers diagram is a difference of two Ferrers diagram; a particularly interesting class of skew Ferrers diagrams are parallelogram polyominoes, studied for instance in \cite{Polya1,SkewFer}.

\begin{definition}
A \emph{parallelogram polyomino} is a skew Ferrers diagram that is connected and has no cut point. We illustrate the terminology in Figure \ref{fig:polyocazzi} below.
\end{definition}

\begin{figure}[!ht]
\begin{tikzpicture}[scale=0.75]
\node[anchor=west] at (4.7,2.4)   
   {\tiny{\ydiagram{1,1+1,1+2}}};
\node[anchor=west] at (7.7,2.4)   
   {\tiny{\ydiagram{1,1+0,1+3}}};
\node[anchor=west] at (10.7,2.4)   
   {\tiny{\ydiagram{2,1+2}}};
\end{tikzpicture} \caption{From left to right: a connected skew Ferrers diagram with a cut point, a disconnected skew Ferrers diagram, a parallelogram polyomino of area $4$.}\label{fig:polyocazzi}
\end{figure}
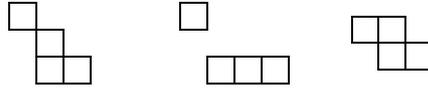

\begin{prop}\label{prop:torusfix}
The $\mathbf T$-fixed locus $\mathcal C(n)^{\mathbf T}\subset \mathcal C(n)$ lies in $\mathcal C(n)_0$ and is finite. The indecomposable $\mathbf T$-fixed modules are in bijection with the set of parallelogram polyominoes.
\end{prop}

\begin{proof}
The support of a torus-fixed module is a torus-fixed subscheme of $\A^2$, so $\mathcal C(n)^{\mathbf T}\subset \mathcal C(n)_0$.
Let $F\in \mathcal C(n)^{\mathbf T}$ be a torus fixed $A$-module. Then $F$ corresponds to a $\mathbf T$-representation
\[
\rho:\mathbf T\ra \GL(V)
\]
of the underlying vector space $V$
such that the maps $\rho_t:V\,\widetilde{\ra}\,V$ are $A$-linear isomorphisms. Let $\Gamma=\Z^2$ be the character lattice of the torus. There is a $k$-linear 
decomposition
\[
V=\bigoplus_{\chi\in \Gamma}V_\chi
\]
into irreducible subrepresentations 
\[
V_\chi=\Set{v\in V|\rho_t(v)=\chi(t)v\textrm{ for all }t\in \mathbf T}.
\]
So $V$ has an eigenbasis $\mathbf v=\set{v_1,\dots,v_n}$ indexed by characters $\chi_1,\dots,\chi_n\in\Gamma$. By $A$-linearity of $\rho_t$, for every $i$ we have the relation 
\[
\rho_t(x\cdot v_i)=x\cdot \rho_t(v_i)=\chi_i(t)(x\cdot v_i).
\]
In other words, $x\cdot v_i$ is either $0$ or belongs to $\mathbf v$. The same holds
for $y\cdot v_i$ by the same reasoning. We have shown that a torus fixed module has a $k$-linear basis $\mathbf v$
such that $x\cdot v_i$ and $y\cdot v_i$ both lie in $\mathbf v\cup\{0\}$. The set 
\[
\Set{\chi_1,\dots,\chi_n}\subset \Gamma
\]
determines a skew Ferrers diagram of area $n$, hence $\mathcal C(n)^{\mathbf T}$ is finite. Finally, any cut point of a skew Ferrers diagram determines two proper submodules of the corresponding module, and these are necessarily direct summands. 
Conversely, a decomposable (torus-fixed) module can be represented by joining several parallelogram polyominoes creating cut points.
\end{proof}

\medskip
Consider the numbers 
\[
c_n=\big|\Set{F\in \mathcal C(n)^{\mathbf T}|F\textrm{ indecomposable}}\big|
\]
and their generating function
\[
\sum_{n=0}^\infty c_n q^n=1+q+2q^2+4q^3+9q^4+20q^5+\cdots
\]
The main result of~\cite{SkewFer} is the calculation of the generating function $F$ of the numbers of parallelogram polyominoes with prescribed area and number of columns. The result is
\[
F(t;q)=\sum_{n=0}^\infty\frac{(-1)^nq^{\binom{n+1}{2}}}{(q;q)_n(q;q)_{n+1}}q^{n+1}t^{n+1}\Bigg/\sum_{n=0}^\infty\frac{(-1)^nq^{\binom{n}{2}}}{(q;q)_n^2}q^{n}t^{n},
\]
where $(a;q)_n=\prod_{i=0}^{n-1}(1-aq^i)$ is the $q$-Pochhammer symbol and $q$ (resp.~$t$) is the variable that keeps track of the area (resp.~number of columns). It follows from Proposition~\ref{prop:torusfix} that one can compute
\[
\sum_{n=0}^\infty c_n q^n=F(1;q).
\]

\bigskip

{\noindent{\bf Acknowledgements.}
The first named author is supported by the Department of Mathematics and Natural Sciences of the University of Stavanger in the framework of the grant 230986 of the Research Council of Norway. We wish to thank Martin G.~Gulbrandsen for the many insightful conversations around the topics discussed here. We also thank Jim Bryan, Aurelio Carlucci, Roy Skjelnes and Mattia Talpo for their valuable comments during the preparation of this work.
}

\clearpage
\bibliographystyle{amsalpha}
\bibliography{bib}

\end{document}